\documentclass{amsart}
\usepackage{amsaddr}

\usepackage{amssymb}
\usepackage{bm}
\usepackage{graphicx}
\usepackage{psfrag}
\usepackage{color}
\usepackage{url}
\usepackage{amsmath,color}
\usepackage{enumerate}
\usepackage{algpseudocode}
\usepackage{physics}
\usepackage{tikz}
\usetikzlibrary{angles,quotes, calc}

\usepackage{xy}
\input xy
\xyoption{all}

\newcommand{\excise}[1]{}

\newcommand{\ve}{\boldsymbol}

\newtheorem{thm}{Theorem}
\newtheorem{lemma}[thm]{Lemma}
\newtheorem{cor}[thm]{Corollary}

\theoremstyle{definition}

{\bfseries}{\normalfont}
{\bfseries}{\normalfont}

{\bfseries}{\normalfont}

\newcommand{\be}{\begin{eqnarray}}
\newcommand{\bea}{\begin{eqnarray*}}
\newcommand{\ee}{\end{eqnarray}}
\newcommand{\eea}{\end{eqnarray*}}

\newcommand{\vol}{\mathrm{vol}}
\newcommand{\R}{\mathbb R}
\newcommand{\Z}{\mathbb Z}
\renewcommand{\H}{\mathcal H}

\newcommand{\C}{\mathcal C}
\newcommand{\Q}{\mathcal Q}

\renewcommand{\L}{{\mathcal L}}
\renewcommand{\P}{{\mathcal P}}

\title[Hyperplane sections of $d$-cube]{On the volume of hyperplane sections of a $d$-cube}
\author[I. Aliev]{Iskander Aliev}

\address{Mathematics Institute, Cardiff University, Cardiff, Wales, UK}
\email{alievi@cardiff.ac.uk}

\date{\today }

\begin{document}

\begin{abstract}
\noindent
We obtain an optimal upper bound for the normalised volume of a hyperplane section of an origin-symmetric $d$-dimensional cube.
This confirms a conjecture posed by  Imre B\'ar\'any and P\'eter Frankl.

\end{abstract}

\maketitle

\section{Statement of the results}

Let $\C^d=[-1/2,1/2]^d$ be the $d$-dimensional unit cube. Throughout this paper we assume that $d\ge 2$. For a nonzero vector ${\ve v}\in \R^{d}$
we will denote by ${\ve v}^\bot$ the hyperplane orthogonal to ${\ve v}$ and consider the section $\C^d\cap {\ve v}^{\bot}$ of the cube $\C^d$.
Let further $\|\cdot\|_1$ and $\|\cdot\|_2$ denote $\ell_1$ and $\ell_2$ norms, respectively. In this paper we will be interested in the quantity
\be\label{V_d}
V_d= \max_{{\ve v}\in \R^d} \frac{\|{\ve v}\|_1}{\|{\ve v}\|_2}\cdot \vol_{d-1}( \C^d\cap {\ve v}^{\bot})\,,
\ee
where $\vol_{d-1}(\cdot)$ stands for the $(d-1)$-volume.
Imre B\'ar\'any and P\'eter Frankl \cite{BF} conjectured that the maximum in (\ref{V_d}) is given by the vector ${\ve v}={\ve 1}_d:= (1, \ldots,1 )$. Our main result confirms this conjecture.

\begin{thm}\label{slicing_thm}
We have
\be\label{thm_V_d}
V_d = \sqrt{d}\cdot \vol_{d-1}(\C^d\cap {\ve 1}_d^{\bot})\,.
\ee
\end{thm}
It is known that
\bea \lim_{d\rightarrow\infty}\vol_{d-1}(\C^d\cap  {\ve 1}_d^{\bot})=\sqrt{\frac{6}{\pi}} \eea
(see \cite{La}, \cite{Po} and e. g.
\cite{ChLo}).  
The expression (\ref{thm_V_d}) also allows finding the exact value of $V_d$ in the following way. Let ${\ve
s}=(s_1,\ldots,s_d)\in{\mathbb S}^{d-1}$ be a unit vector. Then
\be \vol_{d-1}(\C^d\cap {\ve s}^{\bot})
=\frac{2}{\pi}\int_0^\infty\prod_{i=1}^d\frac{\sin s_i t}{s_i t}
\dd{t}\label{volume_via_sinc}\ee
(see e. g. \cite{Ba}).  Consider the {\em
sinc} integral \cite{BoBo}
\bea \sigma_d=\frac{2}{\pi}\int_0^\infty\left(\frac{\sin
t}{t}\right)^d \dd{t}\,. \eea
In view of (\ref{thm_V_d}) and (\ref{volume_via_sinc})  we have 
\be \label{central_sinc}
V_d=
\frac{2\sqrt{d}}{\pi}\int_0^\infty\left(\frac{\sin
\frac{t}{\sqrt{d}}}{\frac{t}{\sqrt{d}}}\right)^d  \dd{t}=
d\,\sigma_d\,.\ee
Further 
\bea \sigma_d=\frac{d}{2^{d-1}}\sum\limits_{0\le
r<d/2,\,r\in\Z}\frac{(-1)^r(d-2r)^{d-1}}{r!(d-r)!}\,\eea
(see e. g.
\cite{MeRo}). The sequences of numerators and denominators of $\sigma_d/2$ can be
found in \cite{Sloane}. 

Theorem \ref{slicing_thm} and (\ref{volume_via_sinc}) immediately imply the following lower bound for sinc integrals.
\begin{cor}
For any unit vector ${\ve s}=(s_1,\ldots,s_d)\in{\mathbb S}^{d-1}$
\bea \frac{2||{\ve s}||_1}{\pi d}\int_0^\infty\prod_{i=1}^d\frac{\sin s_i
t}{s_i t}\, \dd{t}\le  \sigma_d\,.\eea

\end{cor}

It is known that $0<\sigma_{d+1}/\sigma_{d}< 1$ (see e. g. \cite[Lemma 1]{SL}). 
Theorem \ref{slicing_thm} also implies the following lower bound for  the ratio of consecutive sinc integrals. 
\begin{cor}\label{lower_sinc_ratio}
We have
\bea \frac{d}{d+1}\le \frac{\sigma_{d+1}}{ \sigma_{d}}\,.\eea \label{Borweins}
\end{cor}

\section{Intersection body of $\C^d$}

We can
associate with each 
star body $\L$ the {\em distance function} 
$f_{\L}({\ve x})=\inf\{\lambda>0 : {\ve x}\in \lambda \L\}\,. $
The {\em intersection body}  $I{\L}$ of a star body $\L\subset\R^d$ (recall that we assume $d\ge 2$) is defined as the ${\ve 0}$--symmetric star body with
distance function 
\bea f_{I{\L}}({\ve x})=\frac{\|{\ve x}\|_2}{\vol_{d-1}(
\L\cap {\ve x}^{\bot})}\,.\eea
The Busemann theorem (see e. g. \cite{Ga}, Chapter 8) states that if
${\L}$ is ${\ve 0}$--symmetric and convex, then $I{\L}$ is a convex set.
For more details on intersection bodies we refer the reader to \cite{Ko, Lu}.

For convenience, in what follows we will work with normalised cube 
\bea
\Q^d= \frac{1}{\vol_{d-1}(\C^d\cap {\ve 1}_d^\bot)^{1/(d-1)}} \cdot \C^d\,. 
\eea
Then, in particular,
\be\label{section_volume_one}
\vol_{d-1}(\Q^d\cap {\ve 1}_d^\bot)=1\,.
\ee

\begin{lemma}\label{support}
The affine hyperplane
\bea
\H=\{{\ve x}\in\R^d: x_1+\cdots+x_d=\sqrt{d}\}
\eea
is a supporting hyperplane of $I\Q^d$. 
\end{lemma}

\begin{proof}

Let $f=f_{I\Q^d}$ denote the distance function of $I\Q^d$, so that
$
I\Q^d=\{{\ve x}\in \R^d: f({\ve x})\le 1\}\,.
$
By (\ref{section_volume_one}), for the point
\be\label{h_def}
{\ve h}:=\frac{1}{\sqrt{d}}{\ve 1}_d = \left(  \frac{1}{\sqrt{d}}, \ldots,  \frac{1}{\sqrt{d}} \right)
\ee
we have  $f({\ve h})=1$. Therefore ${\ve h}$ is on the boundary of $I\Q^d$.

Suppose, to derive a contradiction, that $\H$ is not a supporting hyperplane of $I\Q^d$. Observe that ${\ve h}\in \H\cap I\Q^d$. Hence for any $\epsilon>0$
there exists a point ${\ve p}=(p_1, \ldots, p_d)$ in the interior of $I\Q^d$ with 
\be \label{close}
\|{\ve h} - {\ve p} \|_2<\epsilon
\ee  
and $p_1+\cdots+p_d>\sqrt{d}$. 

By  (\ref{close}) we may assume that ${\ve p}\in \R^d_{>0}$. Further, as the point ${\ve p}$ is in the interior of $I\Q^d$ we may assume, for simplicity,  that the entries of ${\ve p}$ are pairwise distinct: $p_i\neq p_j$ for $i\neq j$. Consider $d$ points
\bea
\begin{array}{lllllll}
{\ve p}_1& = &(p_1,& \ldots,& p_{d-1},& p_d)\\
{\ve p}_2& = &(p_2,& \ldots,& p_{d},& p_1)\\
\vdots\\
{\ve p}_d& = &(p_d,& \ldots,& p_{d-2},& p_{d-1})\,.
\end{array}
\eea
For each $i$, the section $\Q^d\cap {\ve p}_i^{\bot}$ is the image of the section $\Q^d\cap {\ve p}_1^{\bot}$ under an orthogonal transformation
defined by a permutation matrix. Therefore ${\ve p}_i\in I\Q^d$. Set
\bea
{\ve y}= \frac{1}{d} ({\ve p}_1+\cdots+{\ve p}_d) 
= \frac{\sum_{i=1}^d p_i}{\sqrt{d}} {\ve h}\,.
\eea
By construction, ${\ve y}$ is a convex combination of the points ${\ve p}_1, \ldots, {\ve p}_d$.
Since $I\Q^d$ is convex, ${\ve y}=(y_1, \ldots, y_d)\in I\Q^d$. Further
\bea
y_1+\cdots+y_d= \sum_{i=1}^d p_i>\sqrt{d}\,.
\eea
Therefore the point ${\ve h}$ must be in the interior of $I\Q^d$. The derived contradiction completes the proof.

\end{proof}

\section{Proof of Theorem \ref{slicing_thm}}

%
%
%
It is sufficient to show that for any unit vector ${\ve v}\in {\mathbb S}^{d-1}$ the inequality
\be\label{unit_ineq}
\|{\ve v}\|_1\cdot \vol_{d-1}(\Q^d\cap {\ve v}^\bot)\le \sqrt{d}\cdot \vol_{d-1}(\Q^d\cap {\ve 1}_d^\bot)=\sqrt{d}
\ee
holds.

In view of symmetry of $\Q^d$ we may assume without loss of generality that ${\ve v}\in\R^d_{\ge 0}$.
%
Consider the plane $\P$ spanned by the vector ${\ve h}$, defined by (\ref{h_def}), and the vector ${\ve v}$  and let $\alpha$ be the angle between these two vectors
with $\cos(\alpha)= {\ve h}\cdot {\ve v}$ (Figure \ref{plane_geometry_figure}). It is not difficult to see that $\cos(\alpha) \ge 1/\sqrt{d}$ and, consequently, $\alpha<\pi/2$.

\begin{figure}[!h] 

\centering

\begin{tikzpicture}

\def\myangle{63.43};

\draw (0,0) circle (2cm);
\draw (0,0)--(2,0) node[right]{${\ve h}$};
\node at (0,0) {\textbullet};
\node [anchor=north] at (0,0) {${\ve 0}$};

\draw (2,5) -- (2,-2);
\draw (0,0) -- (2,4);
\draw (0.8944,5) -- (0.8944,-2);

\coordinate (A) at (0,0);
\draw [->] ($(A)+(0.4,0)$) arc (0:\myangle:0.4cm);

\node [anchor=east] at (-2,0) {${\mathbb S}^{d-1}\cap \P$};

\node [anchor=north] at (0.5,0.5) {$\alpha$};
\node [anchor=north] at (0.7,2.45) {${\ve v}$};
\node [anchor=north] at (1.2,0.0) {${\ve w}$};
\node [anchor=west] at (2.0,4.0) {${\ve u}$};
\node [anchor=south west] at (2.0,-2.0) {$\H\cap \P$};

\end{tikzpicture}

\caption{Geometric argument on the plane $\P$}

\label{plane_geometry_figure}

\end{figure}
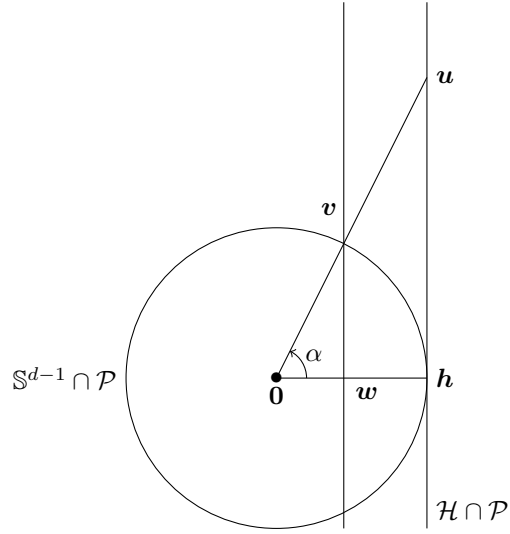

Notice that ${\ve h}$ is orthogonal to the line $\H\cap \P$. Let  ${\ve u}$ denote the intersection point of the line spanned by ${\ve v}$ and  $\H\cap \P$. Further, let ${\ve w}$ be  the orthogonal projection of ${\ve v}$ onto the line spanned by ${\ve h}$. 

Then we have
\be\label{cos}
\cos({\alpha})=\|{\ve w}\|_2=\frac{1}{\|{\ve u}\|_2}\,.
\ee
Since ${\ve h}\in \H$, all points  ${\ve x}$ on the line passing through the points ${\ve v}$ and ${\ve w}$ have $x_1+\cdots+ x_d=\sqrt{d}\, \|{\ve w}\|_2$.  Therefore, we have $\|{\ve v}\|_1=\sqrt{d}\,\|{\ve w}\|_2$. In was shown in Lemma \ref{support} that $\H$ is a supporting hyperplane of $I\Q^d$. Hence we have 
\be\label{below_hyperplane}
\vol_{d-1}(\Q^d\cap {\ve v}^\bot)\le \|{\ve u}\|_2\,.
\ee
Finally, using (\ref{cos}) and (\ref{below_hyperplane}), we have
\bea
\|{\ve v}\|_1\cdot\vol_{d-1}(\Q^d\cap {\ve v}^\bot)\le \sqrt{d}\, \|{\ve w}\|_2\|{\ve u}\|_2= \sqrt{d}\,,
\eea
that confirms (\ref{unit_ineq}).

\section{Proof of Corollary \ref{lower_sinc_ratio}}

It was observed in \cite{BF} that the sequence $\{V_d\}_{d=1}^{\infty}$ is increasing: $V_d\le V_{d+1}$ for all $d\ge 2$.
It is now sufficient to note that, by Theorem \ref{slicing_thm} (see  (\ref{central_sinc})), we can write $V_d=d \sigma_d$.

\section{Acknowledgements}

The author wishes to thank Gergely Ambrus, Imre B\'ar\'any and Martin Henk for valuable comments and suggestions.

\end{document}